\documentclass[12pt,letterpaper]{amsart}
\usepackage{amsmath, amsthm, amssymb}
\usepackage[tmargin=1.4in,bmargin=1.2in,rmargin=1.4in,lmargin=1.4in]{geometry}
\usepackage[breaklinks=true]{hyperref}

\usepackage{soul} 

\usepackage{changes} 

\theoremstyle{plain}
\newtheorem{theorem}{Theorem}[section]
\newtheorem{lemma}{Lemma}[section]

\theoremstyle{definition}
\newtheorem{defin}{Definition}[section]

\DeclareMathOperator{\Aut}{\ensuremath{Aut}}

\DeclareMathOperator{\ad}{\ensuremath{ad}}
\DeclareMathOperator{\rad}{\ensuremath{rad}}

\begin{document}
\title{Center and derivations of generalized Weyl algebras over $\mathbb{Z}/p^n\mathbb{Z}$}
\author{Ruben Mamani-Velasco, Akaki Tikaradze}
\email{RubenLimbert.MamaniVelasco@rockets.utoledo.edu, Akaki.Tikaradze@utoledo.edu }
\address{University of Toledo, Department of Mathematics \& Statistics, 
Toledo, OH 43606, USA}

\begin{abstract}
Let $A$ be either a classical generalized Weyl algebra (also known as a noncommutative deformation of type A Kleinian singularity) or
the enveloping algebra $U(\mathfrak{sl}_{2})$ over $\mathbb{Z}/p^n\mathbb{Z}.$ In this paper we compute the center and derivations of $A.$ More specifically,
we show that the center of $U(\mathfrak{sl}_2)$ is generated by the Casimir element over the ring of the Witt vectors (of length $n$) of its $p$-center.
Our description of derivations of $A$ implies that if the ground ring is a field $\bold{k}$ of characteristic $p>2,$  then
the restriction homomorphism $HH^1_{\bold{k}}(A)\to Der_{\bold{k}}(Z(A), Z(A))$ from the first Hochschild cohomology of $A$ to $\bold{k}$-derivations
of the center is an isomorphism. 
\end{abstract}
\maketitle

\section{Introduction and main results}

We start by recalling the definition of generalized Weyl algebras.

\begin{defin}
Let $R$ be a ring, let $\phi\in \Aut(R)$ and $z\in R$ be a central
element which is not a zero-divisor. Then the corresponding generalized
Weyl algebra $H(R, \phi, z)$ is the algebra generated over $R$ by
generators $x, y$ subject to the following relations:
\[
yx=z, \quad xy=\phi^{-1}(z), \quad xr=\phi^{-1}(r)x, \quad
yr=\phi(r)y, \quad r\in R.
\]
\end{defin} 

Generalized Weyl algebras were introduced by Bavula \cite{B} and by
Lunts--Rosenberg \cite{LR} (under the name of hyperbolic algebras). They
incorporate many important classes of noncommutative algebras, such as
quantized Weyl algebras, the enveloping algebra $U(\frak{sl}_2)$, and
noncommutative deformations of type A Kleinian
singularities of Hodges \cite{H} (which are also called the classical generalized Weyl algebras). 

Derivations (more generally the Hochschild cohomology) of various classes of generalized Weyl algebras have been extensively studied in the literature, but mostly over a ground field of characteristic 0 (see for example \cite{K}, \cite{FSS}, \cite{X}.)
The goal of this paper is to study the center and derivations of these algebras over $\mathbb{Z}/p^n\mathbb{Z}.$

 The description of the center of the enveloping algebra of $\mathfrak{sl}_2$ over a field of positive characteristic is well-known: a classical result going back to Rudakov and Shafarevich \cite{RS}
 showed that it is generated by the Casimir element over the so-called $p$-center. We extend this result over $\mathbb{Z}/p^n\mathbb{Z}.$  Namely, we show that the center is generated by the Casimir element over a subring (which one might call its $p^n$-center) which is isomorphic to the ring of Witt vectors of length $n-1$ over
 the $p$-center, Theorem \ref{center-sl2}. We also prove a similar statement for classical generalized Weyl algebras.

 Next, we turn to derivations of these algebras. At first, recall that these algebras have no outer derivations over a field of characteristic 0.
 The situation is quite different over a field $\bold{k}$ of positive characteristic. Specifically, we give a full description of $HH^1_{\bold{k}}(A)$ (the space of outer derivations of $A$) as a finitely generated module over 
 the center, see Theorem \ref{derivations USL2}.
 As a corollary, we show that the restriction map $HH^1_{\bold{k}}(A)\to Der_{\bold{k}}(Z(A))$  is an isomorphism when $\bold{k}$  is a field. This result is interesting taking into account that
 it is a well-known property of an Azumaya algebra $A$ over a $\bold{k}$-algebra $Z$ that the restriction homomorphism $HH^1_{\bold{k}}(A)\to Der_{\bold{k}}(Z)$  is an isomorphism. However,
 neither $U(\mathfrak{sl}_2)$ nor classical generalized Weyl algebras (for non-generic parameters) are Azumaya algebras.

\section{Center and derivations of $U(\frak{sl}_2)$ over $\mathbb{Z}/p^n\mathbb{Z}.$}
We fix once and for all a perfect field $\bold{k}$  of characteristic $p>2.$
Let $W(\bold{k})$ (respectively $W_n(\bold{k})$)  denote the ring of Witt vectors over $\bold{k}$ (resp. the ring of Witt vectors of length $n$).
Recall that $W_n(\bold{k})$  is uniquely determined up to an isomorphism as a flat commutative $\mathbb{Z}/p^n\mathbb{Z}$-algebra $R$ such that
$R/pR\cong \bold{k}.$ Throughout the paper, we put for simplicity $S=W(\bold{k}), S_n=W_n(\bold{k}).$  Next, we recall some known constructions relating centers and derivations of algebras over $\mathbb{Z}/p^n\mathbb{Z}$ and rings of Witt vectors.

Let $A$ be an $S$-algebra, free as an $S$-module. Put $A_n=A/p^nA$ for $n\in\mathbb{N}$ and $Z_n=Z(A_n)$.
Then one has a ring homomorphism
$$\phi_n:W_n(Z_1)\to Z_n,$$
defined as follows:
$$\phi_n(z_0, z_1,\cdots, z_{n-1})=\sum p^i\tilde{z_i}^{p^{n-i-1}},$$
where $\tilde{z}$  denotes any lift of $z$ to $A_n,$  see \cite{SV} for details.
It is known that the homomorphisms $\phi_n$  are isomorphisms if $A$ is the $m$-th Weyl algebra over $S.$ More
generally, as proved in \cite{SV}, $\phi_n, n\geq 1$ are isomorphisms if $A$ is the ring of differential operators over a smooth affine scheme $X$ over $S.$ 

Next, recall that if $R$ is a commutative ring and $B$ is an $R$-algebra, then  $HH^1_R(B)$  denotes the quotient of $R$-linear derivations of $B$ by the inner ones -- the first Hochschild cohomology of $B$ over $R.$ There is a natural $Z(B)$-module structure on $HH^1_R(B)$. We define the following $S_n$-derivation 
$$d_n:Z_n\to HH^1_{S_n}(A_n),\quad n\geq 1$$  as follows.
Let $z\in Z_n$ and $z'\in A$ be any lift of $z$. Then $\ad(z')=[z', -]$  is a derivation
of $A$ whose image is contained in $p^nA.$  Thus we get a derivation $\frac{\ad(z')}{p^n}:A\to A,$ which descends
to a derivation of $A_n,$ and the image of this derivation in $HH^1_{S_n}(A_n)$  is defined to be $d_n(z).$ One checks easily that
$d_n(z)$  is independent of the choice of a lift $z'.$

Put $A=U(\mathfrak{sl}_2(S)).$  In order to describe our result about the center of $A_n$, we first recall the classical case when
$n=1$ (so $S_1=\bold{k}).$ In this case, the description of the center goes back to Rudakov and Shafarevich \cite{RS}. They showed that the center is generated over
the polynomial subring $\bold{k}[e^p, f^p, h^p-h]=Z_p(\bold{k})$   (the so-called $p$-center) by (a shifted) Casimir element $\Delta=4ef+(h-1)^2$,
subject to the relation 
$$e^pf^p+\frac{1}{4}(h^p-h)^2=\frac{1}{4}\Delta(\Delta^{\frac{p-1}{2}}-1)^2.$$ 
In particular, the center is a free $Z_p(\bold{k})$-module of rank $p$ with basis $1,\Delta,\cdots, \Delta^{p-1}.$

Now, let $n\geq 2.$ Denote by $Z_{p}(S_n)$  the range of the injective homomorphism $\phi_n:W_n(Z_p(\bold{k}))\to Z_n$ defined above.
So, $Z_p(S_n)$ is isomorphic to $W_n(Z_p(\bold{k}))$  and is generated as a subring of $Z_n$ 
over $W_n(\bold{k})$  by the following elements
$$p^{n-i-1}(e^p)^{p^i},\quad p^{n-i-1}(f^p)^{p^i},\quad p^{n-i-1}(h^p-h)^{p^i},\qquad 0\leq i<n.$$
Now we state our result describing the center of $A_n, n\geq 1$ (recall that $A=U(\mathfrak{sl}_2(S)).$ 

\begin{theorem}\label{center-sl2}
 The ring $Z_n$ is generated over $Z_{p}(S_n)$ by $\Delta$ subject to the following relations
 $$p^{n-i-1}(\frac{1}{4}\Delta(\Delta^{\frac{p-1}{2}}-1)^2)^{p^i}=p^{n-i-1}(e^pf^p+\frac{1}{4}(h^p-h)^2)^{p^i}, \quad 0\leq i <n.$$
 In particular, $Z_n$ is generated by $\Delta^i, 0\leq i< p^{n}$ as a $Z_p(S_n)$-module. 
\end{theorem}

Next, we state our description of derivations of  $A_n.$ We set $H=h-1.$ 

\begin{theorem}\label{derivations USL2} 
$HH^1_{S_n}(A_n)$  is generated as a $Z_n$-module by $d_n(Z_n)$ 
and by derivations $D_i, \hat{D}_i,\tilde{D}_m,\bar{D}_m$ where $0\leq i\leq n-1, 1\leq m\leq p^{n-1}$ defined below:

\begin{align*}
    &D_i(f)=0, \;\;\; D_i(h)=p^i(H^p-H)^{p^{n-i-1}},\\
    &D_i(e)f=\frac{1}{2}\left[p^i\Delta^{\frac{p^{n-i-1}+1}{2}}(\Delta^{\frac{p-1}{2}}-1)^{p^{n-i-1}}-p^i(H^p-H)^{p^{n-i-1}}H\right],
\end{align*}

\begin{align*}
    &\hat{D}_i(f)=0, \;\;\; \hat{D}_i(h)=p^i(\Delta^{\frac{p-1}{2}}-1)^{p^{n-i-1}},\\
    &\hat{D}_i(e)f=\frac{1}{2}\left[p^i(H^p-H)^{p^{n-i-1}}-p^i\Delta^{\frac{{p^{n-i-1}}-1}{2}}(\Delta^{\frac{p-1}{2}}-1)^{p^{n-i-1}}H\right],
\end{align*}

\[
\tilde{D}_m(f)=\tilde{D}_m(h)=0,\;\; \tilde{D}_m(e)=f^{mp-1},
\]
and
\[
\bar{D}_m(e)=\bar{D}_m(h)=0,\;\; \bar{D}_m(f)=e^{mp-1}.
\]

\end{theorem}

Recall that $A$ (and its reductions $\mod p^n$) can be identified with a generalized Weyl algebra $H(R,\phi,z)$, where $R=S[\Delta, h]$, $x=f, y=e, 
z=ef=\frac{1}{4}(\Delta-(h-1)^2)$ and the automorphism $\phi$ is given as follows: $\phi(h)=h-2, \phi(\Delta)=\Delta.$ Throughout, we are using the standard grading on a generalized Weyl algebra: 
$$\deg(f)=1,\quad  \deg(e)=-1,\quad \deg(r)=0,\quad r\in R.$$

To prove Theorem \ref{center-sl2}, we start by describing central elements of degree 0. Denote by $A^l_n$ the submodule of degree $l$ elements of $A_n.$ Put $Z^l_n=Z_n\cap A_n^l.$ Let $r=\sum_{i=0}^m \alpha_i(\Delta) h^i\in Z_n^0=(A_n^0)^{\phi}$. So
\[
0=\phi^{-1}(r)-r=\sum_{i=0}^m \alpha_i [(h+2)^i-h^i] =\sum_{i=0}^m\alpha_i \sum_{j=0}^{i-1}\binom{i}{j} 2^{i-j}h^j.
\]
It follows that $2m\alpha_m=0$. Thus $p^i|\alpha_m$ and $p^{n-i}|m$ for $i\in\{0,\ldots,n-1\}$ assuming $\alpha_m\neq 0$. Let $\alpha_m=p^i\beta_m$ and $m=lp^{n-i}$ then,  $r-\beta_mp^i(h^p-h)^{lp^{n-i-1}}$ is a central element of $h$-degree less than $m$. Iterating this argument, we obtain
\begin{equation}\label{eq: fix by phi}
   Z_n^0=\sum_{i=0}^{n-1}p^i S_n[\Delta,(h^p-h)^{p^{n-i-1}}]=\sum_{i=0}^{n-1}p^{n-i-1} S_n[\Delta,(h^p-h)^{p^{i}}].
\end{equation}
In particular, for any $j\in\{0,...,n-1\}$:
\begin{equation}\label{eq: p^jR}
    p^jZ_n^0=\sum_{i=j}^{n-1}p^i S_n[\Delta,(h^p-h)^{p^{n-i-1}}]=\sum_{i=j}^{n-1}p^{j+n-i-1} S_n[\Delta,(h^p-h)^{p^{i-j}}].
\end{equation}

Now, we consider homogeneous central elements of positive degree. Let $rf^m\in Z_n^m$, where $r\neq 0\in A_n^0$. Then $[f,rf^m]=0$ implies $r\in Z_n^0$. On the other hand, for any $s\in A_n^0$, we have $[s,rf^m]=0$ which implies $r(\phi^{-m}(s)-s)=0$. In particular $0=r(\phi^{-m}(h)-h)=2rm$. Since $(\phi^{-m}(s)-s)\in 2mA_n^0$, $rm=0$ also implies $r(\phi^{-m}(s)-s)=0$. Hence $p^j|r$ and $p^{n-j}|m$ for all $j\in\{0,...,n-1\}$. 

It follows from equation \ref{eq: p^jR} that central elements of positive degree $m$ such that $p^{n-j}|m$, where $j\in\{0,...,n-1\}$ and $p^{n-j+1}$ does not divide $m$ unless $j=0$, can be expressed as
\begin{equation}\label{eq: pos central elm}
    \sum_{i=j}^{n-1}p^i S_n[\Delta,(h^p-h)^{p^{n-i-1}}]f^{m}=\sum_{i=j}^{n-1}p^{j+n-i-1} S_n[\Delta,(h^p-h)^{p^{i-j}}]f^{m}.
\end{equation}

The same argument works for homogeneous elements of negative degree.

Clearly, any element of the form $\alpha(\Delta)p^i(h^p-h)^{p^{n-i-1}}$ is in $Z_p(S_n)[\Delta]$ by taking the Witt vector $(h^p-h)$ on the $i$-th entry and 0 elsewhere. Hence $Z_n^0\subset Z_p(S_n)[\Delta]$.

Now, consider an element of $Z_n^m, m>0$ and let $j\geq0$ be such that $p^j$ is the highest power of $p$ that divides $m$. From equation \eqref{eq: pos central elm}, for fixed $i\geq0$ we want to show $z=p^l(h^p-h)^{tp^i}f^m\in Z_p(S_n)$ for any $l,t\geq 0$ such that $l\leq n-1$, $i+l\geq n-1$, $j+l\geq n$ and $p\nmid t$.


We are going to apply induction on $i$. If $i=0$, then $l=n-1$ and $p|m$; consequently $p^{n-1}(h^p-h)^tf^m$ is clearly in $Z_p(S_n)$. Suppose it holds for any integer in $[0,i)$; we will prove it for $i$. Note that the element $\hat{z}=p^l\left((h^p-h)^{tp^{i-(n-l-1)}}f^{\frac{m}{p^{n-l-1}}}\right)^{p^{n-l-1}}$ belongs to $Z_p(S_n)$, and more importantly note that $\hat{z}-z\in Z_n^m$ such that the highest power of $h$ is smaller than $tp^{i+1}$ (highest power of $h$ in $z$ or $\hat{z}$). Hence by induction $\hat{z}-z\in Z_p(S_n)$ and we conclude that $Z_n$ is generated by
$\Delta$ over  $Z_p(S_n).$ Thus, to conclude the proof of Theorem \ref{center-sl2}, we need to verify that the relations given in the theorem are generating relations
of $\Delta$ over $Z_p(S_n).$ As before, we proceed by induction on $n.$

For $n=1$, the following equality is well-known \cite{RS}

\begin{equation}\label{eq:e^pf^p}
    e^pf^p=\frac{1}{4}\left(\Delta(\Delta^{\frac{p-1}{2}}-1)^2-(h^p-h)^2\right).
\end{equation}
Therefore, the desired relations hold for all $n\geq 1.$ It remains to show that they are in fact generating relations. To do this, it suffices to show
that the elements $1, \Delta,\cdots, \Delta^{p^{n}-1}$  are linearly independent in $U(\mathfrak{sl}_2(\bold{k}))$  over  $Z_{p}(\bold{k})^{p^{n-1}}.$ 
This easily follows from the corresponding statement on the level of the associated graded algebra of $U(\mathfrak{sl}_2(\bold{k})).$ 
This concludes the proof of Theorem \ref{center-sl2}.\qed\bigskip

Now we prove Theorem \ref{derivations USL2}. First, note that given a derivation $D$ of $A_n$ of degree $m\in \mathbb{Z}$, we have
\[
2mD(e)=[D(h),e],\; 2mD(f)=[D(h),f],\quad 2mD(h)=[D(h),h].
\]
Thus, if $p$ does not divide $m,$ then $D$ is inner. So, without loss of generality we assume that $p|m.$  Also, if $p^n|m,$  then $D(h)$ 
belongs to the center.
Note that $\{1,h,...,h^{p^n-1}\}$ generates $A_n^0$ as a  $Z^0_n$-module. It follows easily by induction that for all $m\in\mathbb{N}$,

\begin{equation*}
    h^mf\in\sum_{i=1}^{\lfloor (m+1)/p \rfloor}S_nh^{ip-1}f+[A_n^0,f].
\end{equation*}

Hence, for any $l>0,$  we have
\begin{equation}\label{rf equation}
    A_n^l=\sum_{i\geq 1}S_n[\Delta]h^{ip-1}f^l+[A_n^{l-1}, f]=\sum_{1\leq i\leq p^{n-1}} S_n[\Delta,(h^p-h)^{p^{n-1}}]h^{ip-1}f^l + [A_n^{l-1}, f].
\end{equation}

Note that for each $1\leq i\leq p^{n-1}$ there exists some $j$ such that $h^{ip-1}f$ can be written as  $d_n(p^{j}(h^p-h)^i)(f)+\sum_{k< ip-1}\alpha_kh^kf,\:\alpha_k\in S_n$. So,
\begin{equation}\label{eq:degree 0 usl2}
    A_n^1=Z_n^0d_n(Z_n^0)(f)+[A_n^0,f],
\end{equation}
where 
\[
Z_n^0d_n(Z_n^0)(f)=\{\sum_{i=1}^{t} (a_id_n(z_i))(f):t\in\mathbb{N}, a_i,z_i\in Z_n^0\}.
\]
Let us assume that $D$ is a derivation of degree $pm, m\geq 0.$ 
It follows from equation~\eqref{rf equation} that after modifying $D$ with a suitable inner derivation, we may write
\begin{equation}\label{eq:D(f)=}
    D(f)=\sum_{i\geq 1}z_ih^{ip-1} f^{pm+1},\quad z_i\in S_n[\Delta].
\end{equation}
Let $i$ be the largest integer such that $z_i\neq 0.$ 
Let $p^l$ be the largest power of $p$ dividing $i.$ 
We claim that $z_i(h^{p}-h)^if^{pm}\in Z_{l+1}.$ This is equivalent to the assertion that $pi|pmz_i$ (since  $(h^p-h)^i\in Z_{l+1}).$ 

To verify this claim, we embed $A_n$ in a localized Weyl algebra first, and verify the corresponding statement in the Weyl algebra, which becomes straightforward. 

Recall that after localizing $f$, the algebra $A_n[f^{-1}]$  is isomorphic to the localized Weyl algebra over
$S_n[\Delta]$ (with generators $x, x^{-1}, y)$  with the isomorphism given by $f\mapsto x$ and $h\mapsto 2yx$.


\begin{lemma}\label{derivationWeylAlgebra}
    Let $D$ be a derivation of the localized Weyl algebra over a commutative ring $L.$  
     Let $D(x)=\sum_{t=0}^m y^tg_t(x), g_t(x)\in L[x].$ 
     Then $t+1$ divides $g'_t(x)$ for all $t.$  
\end{lemma}
\begin{proof}
Put $D(y)=\sum_my^mf_m(x).$ Then applying the derivation $D$ to  $[x, y]=1,$  we get
    \[
    \sum_ty^tg'_t(x)=-\sum_t (t+1)y^tf_{t+1}(x).
    \]

    Now the desired conclusion follows.
\end{proof}

Using the above isomorphism we may extend $D$ to a derivation of the localized Weyl algebra (which we again denote by $D.$)
So, equation \eqref{eq:D(f)=} becomes
\[
D(x)=y^{ip-1}z_ix^{p(i+m)}+\sum_{t<ip-1}y^tg_t(x), 
\]
Using Lemma \ref{derivationWeylAlgebra}, we conclude that $pi|pmz_i,$  as desired.
Hence, by repeatedly subtracting from $D$ derivations of the form $\frac{1}{pi}\ad(z_i(h^p-h)^if^{pm})$  (which belong to $d_i(Z_i)$), 
we may assume throughout the rest of the proof that $D(f)=0.$

We start with the case $\deg(D)=0.$ 
This implies $[D(e),f]=D(h)\in Z_n^0$. Next, we describe the intersection $[A_n^{-1},f]\cap Z^0_n$.

Recall that $[r, f]=\tilde{r}f$, where $\tilde{r}=r-\phi^{-1}(r)$, for any $r\in A_n^0$. 
We next record the following simple result.

\begin{lemma}\label{derivative}
 Let $r\in A_n^0.$  Then $\tilde{r}\in [A_n^{-1}, f]$  if and only if  $r\in (\Delta-H^2)+Z_n^0$.   
\end{lemma}
\begin{proof}
By localizing at $f$ we have $\tilde{r}\in [(r+Z_n^0)f^{-1},f]$. So, $\tilde{r}\in [A_n^{-1}, f]$  if and only if there is some $z_0\in Z_n^0$ such that $(r+z_0)f^{-1}\in A_n^{-1},$ or equivalently $r\in A_n^{-1}f+Z_n^0=(\Delta-H^2)+Z_n^0$.
\end{proof}

Next, we utilize the following.
\begin{lemma}\label{commutator intersection center}
For all $0\leq i<n$, we have
    \[
    [A_n^{-1},f]\cap Z_n^0=(p^i\Delta^{\frac{p^{n-i-1}-1}{2}}(\Delta^{\frac{p-1}{2}}-1)^{p^{n-i-1}},p^i(H^p-H)^{p^{n-i-1}}).
    \]
\end{lemma}

\begin{proof}
It is easy to see that for all $z\in Z_n^0, \widetilde{zH}=-2z$. So,  to show that the right hand side is a subset of the left one, it suffices to check using Lemma \ref{derivative} that $p^i\Delta^{\frac{p^{n-i-1}-1}{2}}(\Delta^{\frac{p-1}{2}}-1)^{p^{n-i-1}}H$ and $p^i(H^p-H)^{p^{n-i-1}}H $ are in $(\Delta-H^2)+Z_n^0$ for all $0\leq i<n,$ which is immediate. 

Conversely, suppose $g(\Delta)\in [A_n^{-1},f]$, for some $g(\Delta)\in S_n[\Delta].$ Then by Lemma \ref{derivative} we have
$g(\Delta)H\in (\Delta-H^2)+Z_n^0$. Hence we may write 
\begin{align*}
    g(\Delta)H&=\sum_{i<n,j} g_{i,j}(\Delta)p^i(H^p-H)^{jp^{n-i-1}} \mod (\Delta-H^2), \\
    &=\sum_{i=1}^{n-1}g_i(\Delta,H^2)p^iH^{p^{n-i-1}-1}(H^{p-1}-1)^{p^{n-i-1}}H+g_0(\Delta,H^2)\mod (\Delta-H^2)\\
    &=\sum_{i=1}^{n-1}g_i(\Delta)p^i\Delta^{\frac{p^{n-i-1}-1}{2}}(\Delta^{\frac{p-1}{2}}-1)^{p^{n-i-1}}H+g_0(\Delta)\mod(\Delta-H^2).
\end{align*}
    Thus, $g(\Delta)\in (p^i\Delta^{\frac{p^{n-i-1}-1}{2}}(\Delta^{\frac{p-1}{2}}-1)^{p^{n-i-1}})$.
\end{proof}

Hence, we have
\begin{multline}\label{eq:D(e)f}
    D(e)f\in (p^i(H^p-H)^{p^{n-i-1}}H-p^i\Delta^{\frac{p^{n-i-1}+1}{2}}(\Delta^{\frac{p-1}{2}}-1)^{p^{n-i-1}},\\ p^i\Delta^{\frac{p^{n-i-1}-1}{2}}(\Delta^{\frac{p-1}{2}}-1)^{p^{n-i-1}}H-p^i(H^p-H)^{p^{n-i-1}})Z_n^0.
\end{multline}

Therefore, by subtracting from $D$ a linear combination of derivations $\hat{D}_i, D_i$ from Theorem \ref{derivations USL2}, we may assume $D(f)=D(h)=0$. This immediately gives us $[D(e),f]=0$, which is equivalent to $D(e)f\in Z_n^0$. 

Our next goal is to describe $A_n^{-1}f\cap Z_n^0$. Let
\[
\tau_i=p^i\left[(\Delta^{\frac{p-1}{2}}-1)^{2p^{n-i-1}}\Delta^{p^{n-i-1}}-(H^p-H)^{2p^{n-i-1}} \right], 0\leq i\leq n-1.
\]

\begin{lemma}\label{thm:D(f)=D(h)=0}
  We have for all $n\geq 1$:
  \[
    A_n^{-1}f\cap Z_n^0=\sum_{i=0}^{n-1}\tau_i S_n[\Delta,(H^p-H)^{p^{n-i-1}}].
    \]
\end{lemma}

\begin{proof}
First, recall that $4ef=\Delta-H^2$ and note that the intersection $A_n^{-1}f\cap Z_n^0=(\Delta-H^2)\cap Z_n^0$ can be realized as the kernel of the ring homomorphism $\theta:Z_n^0\to S_n[H]$ such that $\Delta \mapsto H^2$ and $H\mapsto H$. So, now the reverse inclusion is trivial since each $\tau_i=0\mod (\Delta-H^2)$.

For the direct inclusion we proceed by induction on $n$. For $n=1,$ we have an isomorphism

\begin{equation}\label{eq:height of ker}
    \bold{k}[\Delta,H^p-H]/\ker(\theta)\cong Im(\theta)\subset \bold{k}[H].
\end{equation}

We want to show that $\ker(\theta)$ is a principal ideal generated by $\tau_0=(\Delta^{\frac{p-1}{2}}-1)^2\Delta-(H^p-H)^2$. Indeed, from equation \eqref{eq:height of ker} it follows that $\ker(\theta)$ is a prime ideal of height 1 ($H^p-H\in Im(\theta)$ and $\tau_0\in\ker(\theta)$).  Since $\bold{k}[\Delta,H^p-H]$ is Noetherian and a UFD, then $\ker(\theta)$ must be a principal ideal. In addition $\tau_0=(\Delta^{\frac{p-1}{2}}-1)^2\Delta-(H^p-H)^2$ is clearly irreducible over $(\bold{k}[\Delta])[H^p-H]$. Hence, $\ker(\theta)=(\tau_0)$.

Next, we assume it holds for natural numbers smaller than $n$ and we prove for $n$. By reduction $\mod p$ we get the map
\[
\bar{\theta}:\bold{k}[\Delta,(H^p-H)^{p^n-1}]\to \bold{k}[H].
\]
It follows from the case $n=1$, that $\ker(\bar{\theta})=\tau_0\bold{k}[\Delta,(H^p-H)^{p^n-1}]$. Let $z\in \ker(\theta)$, then $z=\tau_0w+pz'$ for some $w\in S_n[\Delta,(H^p-H)^{p^{n-1}}]$ and $z'\in S[\Delta,H]$. Since $\tau_0 w$ is already in the kernel, then $pz'\in \ker(\theta)=(\Delta-H^2)\cap Z_n^0$. It follows from equation \eqref{eq: p^jR}, that $z'\mod p^{n-1}\in (\Delta-H^2)R\cap Z^0_{n-1}$. So, by induction 
\[
z'\in \sum_{i=0}^{n-2}\tau'_i S_n[\Delta,(H^p-H)^{p^{n-i-2}}]\mod p^{n-1},
\]
where $\tau'_i= p^i\left[(\Delta^{\frac{p-1}{2}}-1)^{2p^{n-i-2}}\Delta^{p^{n-i-2}}-(H^p-H)^{2p^{n-i-2}} \right]$.
Thus, 
\begin{align*}
    z\in&\tau_0 S_n[\Delta,(H^p-H)^{p^{n-1}}]+\sum_{i=0}^{n-2}p\tau'_i S_n[\Delta,(H^p-H)^{p^{n-i-2}}]\\=&\sum_{i=0}^{n-1}\tau_i S_n[\Delta,(H^p-H)^{p^{n-i-1}}].
\end{align*}

\end{proof}
Since $D(f)=D(h)=0$ and $D(e)f\in \sum_{i=0}^{n-1}\tau_i S_n[\Delta,(H^p-H)^{p^{n-i-1}}],$ we may conclude that $D$ can be expressed as a linear combination of derivations $D_i$ and $\hat{D}_i$ over $Z_n^0$.

Let $\deg(D)=pm>0$ (the case of negative degrees is similar).  We have $[D(h),f]=0$ which implies $D(h)=zf^{pm}$ for some $z\in Z_n^0$. Hence,  $D-\frac{z}{2mp}\ad(f^{mp})$ is a derivation of degree $mp$ that annihilates both $f$ and $h$.
So, without loss of generality, we have $D(f)=D(h)=0$ and $D(e)=rf^{mp-1}$, for some $r\in A_n^0$. Then $[D(e),f]=0$ which implies $r\in Z_n^0$. So, $D=r\tilde{D}_m.$ The case when $D$ has a negative degree is similar. This completes the proof of Theorem \ref{derivations USL2}.

\section{Center and derivations of classical GWA-s}
Let us start by recalling the definition of classical generalized Weyl algebras -- also known as noncommutative deformations of type A Kleinian singularities.

Recall that $S=W(\bold{k}).$ Let $v\in S[h]$ be a nonzero
polynomial. The corresponding algebra $A(v)$ (classical
GWA-noncommutative deformation of type A Kleinian singularity) is defined
as the GWA $H(S[h], \phi, v),$ where
$\phi:S[h]\to S[h]$ is the automorphism given by
translation by 1, so $\phi(f(h))=f(h-1)$. Recall that in this GWA, 
\[
yx=v\mbox{ and } xy=\phi^{-1}(v)=v(h+1).
\]
We recall the following commutator relations:

\begin{equation*}
    [x,y]=v(h+1)-v(h), \; 
[x,h]=x \mbox{ and } [h,y]=y.
\end{equation*}

Put $A=A(v)$  for the remainder of this section. Assume that $v$ is a monic polynomial and $\deg(v)<p.$ 
Given a nonzero polynomial $\alpha\in\bold{k}[h],$ we define  $\rad(\alpha):=\frac{\alpha}{gcd(\alpha, \alpha')}.$ Thus, it follows from our assumption that
$v|\rad(v)^{p-1} \mod p.$ 
Put $v_p=\prod_{i=0}^{p-1}v(h+i)\mod p\in\bold{k}[h^p-h].$ Then the elements $p^i(\rad(v_p))^{p^{n-i-1}}, 1\leq i\leq n$ can and will be viewed as
elements of $S_n[h^p-h].$ 
It is well-known that $Z_1$ is generated by $x^p, y^p, h^p-h$ over $\bold{k}$  subject to the relation $x^py^p=v_p$ (see for example \cite[Proposition~4.4]{BC}).  
 
The following is the main result of this section.

\begin{theorem}\label{derivation of GWA}
The homomorphism $\phi_n:W_n(Z_1)\to Z_n$ is an isomorphism. We have
$$HH^1_{S_n}(A_n)=Z_nd_n(Z_n)+\sum_{i=0}^{n-1}Z_nD_i,$$
where $D_i, 0\leq i\leq n-1$ are the following derivations

\[
D_i(x)=0, \;\;\; D_i(h)=p^i(\rad(v_p))^{p^{n-i-1}}\mbox{ and } D_i(y)=\frac{v'p^i(\rad(v_p))^{p^{n-i-1}}}{v}y.
\]
\end{theorem}

The description of the center is similar (in fact, simpler) to the proof of Theorem \ref{center-sl2}. It also follows from \cite[Theorem 4]{T}. So, we turn to the proof
of the part concerning the derivations.

 Let $D\in Der_{S_n}(A_n)$ be a homogeneous derivation of degree $m\in\mathbb{Z}$, then
\begin{equation*}\label{eq:mD=[D(h),-]}
    mD(x)=[D(h),x],\;mD(y)=[D(h),y], \mbox{ and } mD(h)=[D(h),h].
\end{equation*}

Just as in the proof of Theorem \ref{derivations USL2}, we may assume without loss of generality that $D$ is a homogeneous derivation of
degree $pm$ and $D(x)=0.$ We start with the case $\deg(D)=0.$ 

We have $D(y)=r(h)y$ for some $r(h)\in S_n[h].$  Thus applying $D$ to $yx=v,$  we get $r(h)v=v'D(h).$ This leads to the following.

\begin{lemma}\label{lemma:case D(x)=0}
 For all $n\geq 1$ the following equality holds   \[
S_n[h]v\cap v'Z_n^0=v'\sum_{i=0}^{n-1} p^{i} (\rad(v_p))^{p^{n-i-1}}S_n[(h^p-h)^{p^{n-i-1}}]
\]
\end{lemma}

\begin{proof}
Since $v$ divides $v_p\mod p,$  it follows that $v|\rad(v_p)^{p-1} \mod p.$ 
We proceed by induction on $n.$ Let $n=1.$ 
 Let $v(h)=\prod_{i}(h-\alpha_i)^{m_i}, \alpha_i\in \bar{\bold{k}}$. This implies $v_p=\prod_i(h^p-h-(\alpha_i^p-\alpha))^{m_i}$. Hence, $\rad(v)=\prod_{i}(h-\alpha_i)$ and $\rad(v_p)=\prod_i(h^p-h-(\alpha_i^p-\alpha))^{k_i}$ with $k_i\in\{0,1\}$. It follows that $(\rad(v_p))\subset (\rad(v))$, in particular $v$ divides $v'\rad(v_p)$ and
 $v'\rad{v_p}\in v\bold{k}[h]\cap v'\bold{k}[h^p-h].$ Now let $g(h^p-h)\in\bold{k}[h^p-h]$ be so that $v|gv'.$ This implies that each $\alpha_i$ must be a root of
 $g,$  hence $h^p-h-(\alpha_i^p-\alpha_i)|g.$  So $\rad(v_p)|g.$  This concludes the proof of the desired equality $\bold{k}[h]v\cap Z_1^0=v'\rad(v_p)Z_1^0$  (for $n=1).$

Now, assume that the equality holds for any natural number less than $n$. We first verify that the right hand side is a subset of the left one.
Let $r=p^i(\rad(v_p))^{p^{n-i-1}}$ for some $i\in\{0,1,...,n-2\}$; we want to show that $v|v'r$. By induction $v|v'(p^i(\rad(v_p))^{p^{n-i-2}})\mod p^{n-1}$, and since
 $v\mod p$ divides $[(\rad(v_p))^{p-1}]^{p^{n-i-2}}\mod p$, we are done. On the other hand, if $i=n-1$, $v|v'\rad(v_p)\mod p$, which implies $v|v'p^{n-1}\rad(v_p)$.

For the reverse inclusion, we assume $v|v'g(h),$ where $g(h)\in Z_n^0.$  This implies $\rad(v_p)|g(h)\mod p$. Since 
$g(h)\in (\bold{k}[h^p-h])^{p^{n-1 }}\mod p$, we have $(\rad(v_p))^{p^{n-1}}|g(h)\mod p.$ So, $g(h)=z(\rad(v_p))^{p^{n-1}}+pg_1(h)$ where $z\in S_n[(h^p-h)^{p^{n-1}}]$.
It follows that $v|v'g_1(h)\mod p^{n-1}.$ Thus, by the inductive assumption we have $$g_1(h)\in \sum_{i=0}^{n-2} p^{i} (\rad(v_p))^{p^{n-i-2}}S_n[(h^p-h)^{p^{n-i-2}}]\mod p^{n-1}.$$
Hence $$g(h)\in \sum_{i=1}^{n-1} p^{i} (\rad(v_p))^{p^{n-i-1}}S_n[(h^p-h)^{p^{n-i-1}}].$$ This completes the proof.
\end{proof}

\section{Derivations of the center}

The following is the main result of this section.

\begin{theorem}
Let $A$ be either $U(\mathfrak{sl}_2(S))$  or a classical generalized Weyl algebra $A(v)$  for $v\in S[h].$ 
The restriction homomorphism $HH^1_{S_n}(A_n)\to Der_{S_n}(Z_n)$  is an isomorphism for $n=1$ and is not an isomorphism for $n>1.$ 
\end{theorem}

 In fact, we show more: assuming that $A_1=U\mathfrak{sl}_2(\bold{k})$  (respectively $A_1=A(v), v\in\bold{k}[h]$),
then for each $m\in\mathbb{Z}$  divisible by $p$ the degree $m$ components of  $HH^1_{\bold{k}}(A_1)$  and $Der_{\bold{k}}(Z_1)$  are free modules of rank 3 (respectively rank 2)
over $Z_1^0$, and the restriction homomorphism is an isomorphism of $Z_1^0$ -modules.

We start with the case $A_1=U(\mathfrak{sl}_2(\bold{k})).$ 
We first describe the derivations of the center 
$$Z_1=\bold{k}[e^p,f^p,h^p-h,\Delta]/(e^pf^p-\frac{1}{4}(\Delta(\Delta^{\frac{p-1}{2}}-1)^2-(h^p-h)^2)).$$
 Recall that $Z_1^0= \bold{k}[\Delta, h^p-h]$.  
 Suppose that $D$ is a derivation of $Z_1$ of degree 0. We start by showing that there exists a derivation of degree 0 of $A_1$ which restricts to $D.$ 
 Put $D(e^p)=z_1e^p$ and $D(f^p)=z_2f^p$ for some $z_1,z_2\in \bold{k}[\Delta, h^p-h]$.
 Evaluating $D$ on $e^pf^p$ we get
\begin{equation}\label{eq:D(h^p-h)}
    (z_1+z_2)((\Delta^{\frac{p-1}{2}}-1)^2\Delta-(h^p-h)^2)=(1-\Delta^{\frac{p-1}{2}})D(\Delta)-2(h^p-h)D(h^p-h).
\end{equation}
Hence,
\begin{equation*}
    (z_1+z_2)(h^p-h)^2=2(h^p-h)D(h^p-h)\in(\bold{k}[\Delta]/(\Delta^{\frac{p-1}{2}}-1))[h^p-h].
\end{equation*}
Note that $D(h^p-h)\in \bold{k}[\Delta][h^p-h]$, so from the above equality the constant term of $D(h^p-h)$ is in $(\Delta^{\frac{p-1}{2}}-1).$
So, $D(h^p-h)\in(\Delta^{\frac{p-1}{2}}-1,h^p-h)$. Hence, there exist $z,\hat{z}\in \bold{k}[\Delta,h^p-h],$ such that 
$D(h^p-h)-zD_0(h ^p-h)-\hat{z}\hat{D}(h^p-h)=0,$ where the derivations $D_0, \hat{D}_0$ are defined in Theorem \ref{derivations USL2}.

So, there exists a derivation of $A_1$ which agrees with $D$ on  $(h^p-h).$
Thus, without loss of generality we may assume that $D(h^p-h)=0.$ 
Therefore
\begin{equation*}
    (z_1+z_2)((\Delta^{\frac{p-1}{2}}-1)^2\Delta-(h^p-h)^2)=(1-\Delta^{\frac{p-1}{2}})D(\Delta).
\end{equation*}
Then $D(\Delta)=z_0((\Delta^{\frac{p-1}{2}}-1)^2\Delta-(h^p-h)^2)$ for some $z_0\in Z_1^0.$  
Since $4ef=\Delta-(h^p-h)^2$, then
\[
((\Delta^{\frac{p-1}{2}}-1)D_0-(h^p-h)\hat{D}_0)(4ef)=((\Delta^{\frac{p-1}{2}}-1)D_0-(h^p-h)\hat{D}_0)(\Delta-(h^p-h)^2).
\]
We have
\[
((\Delta^{\frac{p-1}{2}}-1)D_0-(h^p-h)\hat{D}_0)(\Delta)=2((\Delta^{\frac{p-1}{2}}-1)^2\Delta-(h^p-h)^2),
\]
and
\[
((\Delta^{\frac{p-1}{2}}-1)D_0-(h^p-h)\hat{D}_0)(h^p-h)=0.
\]
Hence, replacing $D$ by  $D-\frac{z_0}{2}(\Delta^{\frac{p-1}{2}}-1)D_0-\frac{z_0}{2}(h^p-h)\hat{D}_0,$  we  can assume that $D(\Delta)=D(h^p-h)=0$. From the equation \eqref{eq:D(h^p-h)}:
\[
(z_1+z_2)((\Delta^{\frac{p-1}{2}}-1)^2\Delta-(h^p-h)^2)=0,
\]
then $z_1=-z_2$. So, $D(f^p)=z_1f^p$, $D(e^p)=-z_2e^p$ and $D(\Delta)=D(h^p-h)=0$. This is exactly the restriction on the center of the derivation 
$\frac{z_1}{2}\frac{1}{p}[h^p-h,-]$.
This concludes the proof that any degree-zero derivation of $Z_1$  is obtained by a restriction of a degree-zero derivation of $A_1$. 

 Now, suppose $D$ is a degree $pm$ derivation of $Z_1$  with $m>0$ (the negative degree case is similar).   
 Let 
 $$D(e^p)=z_1f^{(m-1)p},\quad D(f^p)=z_2f^{(m+1)p},\quad D(\Delta)=z_3f^{mp},\quad D(h^p-h)=z_4f^{mp},$$ 
 with $z_i\in Z^0(A).$   
Since $\frac{1}{pm}[f^{mp},h^p-h]=-2f^{mp}$, then $D(h^p-h)+\frac{z_4/2}{pm}[f^{mp},h^p-h]=0$.
    Thus, subtracting from $D$ the restriction of $\frac{z_4/2}{pm}[f^{mp}, -],$  we may assume $D(h^p-h)=0$.
    Moreover, since the derivation $\frac{f^{mp}}{p}[h^p-h,-]$ vanishes on $h^p-h$ and $\frac{f^{mp}}{p}[h^p-h,f^p]=2f^{(m+1)p},$ 
we may further assume $D(h^p-h)=D(f^p)=0$.
    Now applying $D$ to equation \eqref{eq:e^pf^p} we get
\begin{equation*}
    4z_1f^{mp}=(1-\Delta^{\frac{p-1}{2}})z_3f^{mp},
\end{equation*}
thus
\begin{equation*}
    4z_1=(1-\Delta^{\frac{p-1}{2}})z_3, \quad D(e^p)\in (1-\Delta^{\frac{p-1}{2}})\bold{k}[\Delta,h^p-h].
\end{equation*}
On the other hand, we have  $\tilde{D}_m(\Delta)=4(1-\Delta^{\frac{p-1}{2}})f^{mp}$, which implies that $\tilde{D}_m(e^p)=(1-\Delta^{\frac{p-1}{2}})f^{(m-1)p}$. Therefore, $D$ can be expressed as a restriction 
of a central multiple of the derivation $\tilde{D}_m$, where the derivation $\tilde{D}_m$ is defined in theorem \ref{derivations USL2}.

Finally, we need to verify that the restriction homomorphism $HH^1_{\bold{k}}(A_1)\to Der_{\bold{k}}(Z_1)$  is injective. To this end, it is a straightforward computation to show that
for each $m\in\mathbb{Z}$  that is a multiple of $p,$  the $Z_1^0$-submodule of $HH^1(A_1)$  of degree-$m$ derivations is a rank 3 free module, and the restriction homomorphism is injective on it. For example, Theorem \ref{derivations USL2} implies that the degree 0 part of $HH^1(A_1)$ is generated over $Z_1^0$  by $d_1(h^p-h), D_0, \hat{D}_0.$
Moreover, it is a straightforward computation to show that their restrictions  in $Der_{\bold{k}}(Z_1)$  are linearly independent over $Z_1^0.$ A similar argument applies for any
degree component of $HH^1(A_1).$ 

Now, we move to the case of classical generalized Weyl algebras $A_1=A(v), v\in\bold{k}[h].$
Let $D\in Der_{\bold{k}}(Z_1)$ be of degree 0. By considering its action on $y^px^p=v_p(h^p-h),$ 
 we get $D(h^p-h)\in(\rad(v_p))$. Note that the restriction of $D_0$ from Theorem \ref{derivation of GWA} is such that $D_0(h^p-h)=-\rad(v_p)$. So, 
 by subtracting from $D$ an appropriate multiple of the restriction of $D_0,$  we may assume $D(h^p-h)=0$. Note also that the derivation $d_1(h^p-h)$ acting on $x^p$ is $x^p$, and vanishes at $h^p-h$. Hence, we can assume further that $D(h^p-h)=D(x^p)=0.$ But this implies $D(y^p)=0$. Thus, $D$ can be expressed as a 
 restriction of a degree-zero derivation of $A_1$. 
 The case when $\deg(D)\neq 0$ is similar to the proof of the corresponding case for $U(\mathfrak{sl}_2(\bold{k})).$ Thus, we have shown that the restriction homomorphism $HH^1(A_1)\to Der(Z_1)$  is surjective. The proof of injectivity is similar to the corresponding part for $U(\mathfrak{sl}_2(\bold{k})).$  Specifically,
 it is straightforward to deduce from Theorem \ref{derivation of GWA} that all graded pieces of $HH^1(A_1)$  have two generators over $Z_1^0,$  and their restrictions are linearly independent
 in $Der(Z_1).$  This concludes the proof that $HH^1(A_1)$  is isomorphic to $Der(Z_1)$  via the restriction map.

Finally, we show that the restriction map is not an isomorphism when $n>1.$ For simplicity, we consider the case $A_2=A(v), v=h\in S_2[h]$ so that $A_2$ is the first Weyl algebra (the other cases are similar).
We are going to exhibit a degree-zero derivation of $Z_2$  which admits no lift to a degree-zero derivation of $A_2$. 


     We start by observing that given a degree-zero derivation  $D\in Der_{\bold{k}}(A_2)$ we have that $h$ commutes with $D(h),$ so 
     $$D((h^p-h)^p)=p(h^p-h)^{p-1}(ph^{p-1}-1)D(h).$$ Then, $p(D(h^p-h)^p)=0.$ 
     Thus, it suffices to construct a derivation $D$ of the center $Z_2$  such that $pD((h^p-h)^p)\neq 0.$ 
     To this end, it is easy to see that $(h^p-h)^p=y^{p^2}x^{p^2}.$ Moreover, we have the following description of the center
    \[
    Z_2=S_2[x^{p^2}, y^{p^2}]\oplus V,\qquad\text{where}\qquad
    V=\sum_{i,j \not\equiv 0 \mod p}pS_2x^{ip}y^{jp}\]
is an $S_2[x^{p^2},y^{p^2}]$-module such that $V^2=0$. Hence, we can pick
any degree-zero $S$-derivation $D:S_2[x^{p^2}, y^{p^2}]\to S_2[x^{p^2},y^{p^2}]$ such that $pD(x^{p^2}y^{p^2})\neq 0$ and extend it to a derivation of $Z_2$  by putting $D(V)=0.$ 

\end{document}